\documentclass[leqno,12pt]{amsart}

\usepackage{amsmath}
\usepackage{amsthm}
\usepackage{amsfonts}
\usepackage{amssymb}
\usepackage{amsfonts}
\usepackage[cp1250]{inputenc}
\usepackage[T1]{fontenc}
\usepackage{epsfig}
\usepackage{enumerate}
\usepackage{setspace}
\usepackage{xypic}
\usepackage[english]{babel}

\numberwithin{equation}{section}

\theoremstyle{plain} %text of this environment is typesetted in italics
\newtheorem{theorem}{\indent\sc Theorem}[section]
\newtheorem{lemma}[theorem]{\indent\sc Lemma}
\newtheorem{corollary}[theorem]{\indent\sc Corollary}

\theoremstyle{definition} %text of this environment is typesetted in roman letters
\newtheorem{definition}[theorem]{\indent\sc Definition}
\newtheorem{remark}[theorem]{\indent\sc Remark}
\newtheorem{example}[theorem]{\indent\sc Example}

\newtheorem{conjecture}[theorem]{\indent\sc Conjecture}

\newcommand{\CC}{{\mathbb C}}

\newcommand{\RR}{{\mathbb R}}
\newcommand{\QQ}{{\mathbb Q}}

\newcommand{\ZZ}{{\mathbb Z}}
\newcommand{\PP}{{\mathbb P}}

\newcommand{\FFF}{{\mathcal F}}

\newcommand{\OO}{{\mathcal O}}

\newcommand{\ep}{\varepsilon}

\newcommand{\too}{\longrightarrow}
\newcommand{\oot}{\longleftarrow}

\newcommand{\ths}{ \theta\mbox{-stable}}

\newcommand{\modulo}[2]{{\ensuremath{\langle{#1}\rangle_{#2}}}}
\newcommand{\hd}{{\ensuremath{\mathsf{h}}}}
\newcommand{\tl}{{\ensuremath{\mathsf{t}}}}

 \DeclareMathOperator{\Hom}{Hom}
\DeclareMathOperator{\SL}{SL} \DeclareMathOperator{\GL}{GL}
\DeclareMathOperator{\PGL}{PGL} 

\DeclareMathOperator{\Spec}{Spec} 
 \DeclareMathOperator{\diag}{diag}

\DeclareMathOperator{\Rep}{Rep} \DeclareMathOperator{\Wt}{Wt}

\DeclareMathOperator{\Proj}{Proj}

\DeclareMathOperator{\divisor}{div}

\begin{document}

\title[Danilov resolution and representations of McKay quiver]{Danilov resolution and representations of McKay quiver}
\author[Oskar K\k{e}dzierski]{Oskar K\k{e}dzierski}

\address{Institute of Mathematics\endgraf
Warsaw University\endgraf
 ul. Banacha 2\endgraf
 02-097 Warszawa\endgraf
  Poland}

\email{oskar@mimuw.edu.pl} \subjclass[2010]{Primary 14E16; Secondary
16G20,14L24} \keywords{McKay correspondence; resolutions of terminal
quotient singularities; Danilov resolution; moduli of quiver representations}

\thanks{Research supported by a grant of Polish MNiSzW (N N201 2653 33).}

\date{\today}

\maketitle

\begin{abstract}
We construct a family of McKay quiver representations on the Danilov resolution of the
$\frac{1}{r}(1,a,r-a)$ singularity. This allows us to show that the resolution is the normalization
of the coherent component of the fine moduli space of $\theta$-stable McKay quiver representations for a suitable stability
condition $\theta.$ We describe explicitly the corresponding union of chambers of stability conditions for any coprime numbers $r,a.$
\end{abstract}

%\doublespacing

\section{Introduction}

In~\cite{King} King introduced the notion of stability of a quiver representation and via GIT constructed fine moduli of stable representations.
The well-known example of King's moduli is the $G$-Hilbert scheme, i.e. scheme parameterizing all $G$-invariant $0$-dimensional subschemes of $\CC^n$ of length equal to the order of $G$ given for any finite group $G\subset\GL(n,\CC)$. It is the moduli of representations of the McKay quiver, defined by the inclusion of $G$
in the general linear group and representation theory of $G$. The $G$-Hilbert scheme was introduced by Ito and Nakamura in~\cite{IN99:simple} where they proved that it is the minimal (crepant) resolution of $\CC^2/G$ for finite group $G\subset\SL(2,\CC)$ and the relation of the $G$-Hilbert scheme
with the moduli of representations of the McKay quiver was observed by Ito and Nakajima in~\cite{IN:Topology}. The result from~\cite{IN99:simple} was extended by
Bridgeland, King and Reid~\cite{BKR} by showing that the $G$-Hilbert scheme is a crepant resolution of the singularity $\CC^3 /G$ for finite group $G\subset\SL(3,\CC)$. In particular the moduli is smooth and irreducible. Moreover, if $G$ is abelian it turns out that by varying the stability parameter one can get all projective crepant resolutions of $\CC^3/G$, cf.~\cite{CI}. The above results suggest that it may be possible to interpret some other resolutions of quotient singularities as moduli of the McKay quiver representations.

In the following paper we accomplish this task for the Danilov resolution of the $3$-dimensional cyclic terminal quotient singularity. The proof relies on Logvinenko's classification of all natural families (called {\it gnat-families}) of the McKay quiver representations on some fixed resolution $Y$ of $\CC^n/G$ for any finite abelian group $G\subset\GL(n,\CC)$ and his characterization of such families as satisfying {\it the reductor condition}~\cite{Logvinenko:RIMS}. The most obvious candidate to be isomorphic to the resolution is the coherent component, defined by Craw, Maclagan and Thomas as the irreducible component of stable moduli of the McKay quiver representations containing points parameterizing free orbits. In~\cite{CrawMaclaganI} they prove that for a finite, abelian group $G$
it is a possibly non-normal toric variety which admits projective birational morphism to $\CC^n/G$.

Let $G=\langle\diag(\ep,\ep^a,\ep^{r-a})\rangle,$ be a cyclic group of
order $r$ generated by a diagonal matrix, where $\ep=e^{{2\pi
i}/{r}},$ with $a,r$ fixed, coprime natural numbers such that $r>1.$ The
quotient singularity $X=\CC^3 /G$ is the unique $3$-dimensional
cyclic, terminal quotient singularity (cf.~\cite{MorrisonStevens}).
We call this the singularity of type $\frac{1}{r}(1,a,r-a).$
In his proof of weak factorization theorem for toric threefolds, Danilov~\cite{Danilov}
introduced a recursively defined resolution of the singularity of type
$\frac{1}{r}(1,a,r-a)$; this was subsequently named {\it the Danilov resolution} by Reid~\cite{Reid:YPG}.

For $a=\pm 1$ K\k{e}dzierski~\cite{Kedz:GHilb} proved that the Danilov resolution is isomorphic to the component of the $G$-Hilbert scheme
that contains the free $G$-orbits. This cannot hold for $a\neq \pm1$ as the $G$-Hilbert scheme is singular. Our main result (see Theorem~\ref{t:main theorem}) establishes that for $a\neq \pm 1$ the Danilov resolution is isomorphic to the normalization of the coherent component of the moduli space of representations of the McKay quiver for a suitably chosen stability parameter. In fact, we describe precisely the appropriate union of GIT chambers of stability conditions. Finally, we conjecture that the coherent component for such stability conditions is normal.

The main idea is to define {\it a priori} a family of McKay quiver
representations on the Danilov resolution and use the universal property of the moduli space.
To show that the resulting map is injective we prove that any two
representations in the given family are non-isomorphic. This is done by exploiting the recursive nature of the resolution.
On the level of representations, the recursive step can be seen by grouping
the vertices of the McKay quiver in, so-called $L-$bricks and $R$-bricks. Those bricks can be seen as vertices of smaller McKay quivers.

The paper is organized as follows. Section~\ref{section:recursive def of Danilov} recalls definition
of the Danilov resolution. In Section~\ref{s:divisors X,Y,Z} we define effective divisors $X_i,Y_i,Z_i$ and $\QQ$-divisors $R_i$ which
will be used in constructing a family of McKay quiver representations on the Danilov resolution. The McKay quiver is defined in Section~\ref{section:McKay quiver}. The family of quiver representations is constructed in Section~\ref{section:construction of family} and we check that any two representations in that family are non-isomorphic. Section~\ref{section:stability} recalls elementary facts on stability of quiver representations. In Section~\ref{section:some linear algebra} we determine a cone of stability conditions  for the constructed family. Finally, in Section~\ref{section:main theorem} we prove the main theorem
and compute explicitly the union of chambers of stability conditions in the $\frac{1}{5}(1,2,3)$ case.

The author would like to thank Miles Reid for introducing him into this subject and to thank Anonymous Referee for his/her careful reading of the paper and for a great deal of effort put into its improvement. The author is grateful to Alaister Craw for helpful remarks and his help in checking the English grammar.

\section{Notation}\label{section:basic def}

Let $G(r,a)$ be a finite, cyclic subgroup of $\GL(3,\CC)$ generated by a
diagonal matrix $\diag(\ep,\ep^a,\linebreak[1]\ep^{r-a}),$ where
$\ep=e^{{2\pi i}/{r}}$ and $0<a<r$ are coprime numbers. Let $N_0=\ZZ e_1
\oplus \ZZ e_2 \oplus \ZZ e_3$
be a free $\ZZ$-module with basis $e_1,e_2,e_3.$ Denote by
$M_0=\Hom_\ZZ (N_0,\ZZ)=\ZZ e_1^* \oplus \ZZ e_2^* \oplus \ZZ
e_3^*$ the lattice dual to $N_0,$ with $e_i^*(e_j)=\delta_{ij}.$ Define $N(r,a)=N_0+\ZZ\cdot
{1/r}(e_1+ae_2+(r-a)e_3)$
 and $M(r,a)=\Hom_\ZZ (N,\ZZ).$
The lattice $M(r,a)$ can be identified with a sublattice of $M_0,$ consisting of exponents of the $G$-invariant
Laurent monomials.  For any points $p_1,\ldots,p_n$ in the lattice $N(r,a)$ we denote by $\langle p_1,\ldots,p_n\rangle$
the cone spanned by these points. For  a rational polyhedral cone $\sigma$ in $N(r,a)\otimes\RR$ by $U_\sigma$ we mean
the toric chart $\Spec(\CC[\sigma^{\vee}\cap M(r,a)]).$ By $\modulo{s}{t}$ we mean the least non-negative integer $u$ such that
$t$ divides $s-u.$ Sometimes we just write $\modulo{s}{},$ when $t$ is obvious, moreover we write $G, M, N$ instead of $G(r,a), M(r,a)$ and $N(r,a)$ when no confusion arises. All indices and all operations on vertices of the McKay quiver are meant modulo $r.$ By $T$ we mean the torus $\Spec\CC[M].$
The vector $a_1e_1 + a_2 e_2 + a_3 e_3$ will be denoted $(a_1,a_2,a_3).$

\section{Recursive definition of the Danilov
resolution}\label{section:recursive def of Danilov}

Let $r$ and $a$ be coprime, natural numbers, such that $a<r.$ We
recall definition of the Danilov resolution of the singularity $\frac{1}{r}(1,a,r-a)$ (cf.~\cite[p.~381]{Reid:YPG}).
Let $\Delta(r,a)=\langle e_1,e_2,e_3\rangle$
be the positive octant in $N(r,a)\otimes \RR.$ There exists ring
isomorphism of $\CC[\Delta(r,a)^{\vee}\cap M(r,a)]$ with the ring of $G(r,a)$-invariant regular functions on $\CC^3,$
therefore the quotient singularity $X(r,a) = \CC^3/G(r,a)$ is a toric variety given by the cone $\Delta(r,a)$  in the lattice $N(r,a).$

Let $b$ denote the inverse of $a$ modulo $r.$
Set \[p_i=\frac{1}{r}(\modulo{-ib}{r},r-i,i)\quad\text{for\ } i=0,\ldots,r.\]
%Set \[p_i=\frac{1}{r}(\modulo{-ib}{r}e_1+(r-i)e_2+ie_3)\quad\text{for\ } i=0,\ldots,r.\]

Note that by definition $p_{r-a}=1/r(1,a,r-a)$ and $p_0=e_2,\ p_r=e_3.$ The following well-known lemma implies that the toric varieties associated to cones $\langle e_1,e_2,p_{r-a}\rangle, \langle e_1,e_3,p_{r-a}\rangle$
are isomorphic to the quotients of type ${1/(r-a)}(1,\modulo{r}{r-a},\modulo{-r}{r-a})$ and
 ${1/a}(1,\modulo{-r}{a},\modulo{r}{a})$ respectively.

\begin{lemma}
There exist $\ZZ$-linear isomorphisms
\[L(r,a):N(r-a,\modulo{r}{r-a})\too N(r,a),\]
\[R(r,a):N(a,\modulo{-r}{a})\too N(r,a),\]
such that
$L(r,a)(\Delta(r-a,\modulo{r}{r-a}))=\langle e_1,e_2,p_{r-a}\rangle,$
$R(r,a)(\Delta(a,\modulo{-r}{a}))=\langle e_1,e_3,p_{r-a}\rangle,$
and  $L(r,a)(e_1)= R(r,a)(e_1)=e_1,$ $ L(r,a) (e_3)=R(r,a)(e_2)=p_{r-a}.$
\end{lemma}

\begin{definition}
By a weighted blow-up of the singularity of type $\frac{1}{r}(1,a,r-a)$ at the point $p_{r-a}=1/r(1,a,r-a)$ we mean
a toric variety $\overline{X(r,a)}$ obtained by the star subdivision of the cone $\Delta(r,a)$ at the point $1/r(1,a,r-a).$
\end{definition}

The weighted blow-up induces a proper, birational morphism $\overline{X(r,a)}\too X(r,a)$ with the exceptional divisor equal to the weighed
projective space $\PP(1,a,r-a).$

\begin{definition}\label{definiton:danilov resolution}
The Danilov resolution of the singularity of type
$\frac{1}{r}(1,a,r-a)$ is a resolution obtained by the weighted blow-up of the singularity of type $\frac{1}{r}(1,a,r-a)$ at the point $p_{r-a}$ and, recursively, the Danilov resolutions of the singularities of type ${1/(r-a)}(1,\modulo{r}{r-a},\modulo{-r}{r-a})$ and
 ${1/a}(1,\modulo{-r}{a},\modulo{r}{a}).$
\end{definition}

\begin{figure}[ht]
  \centering
  \includegraphics[width=0.5\textwidth]{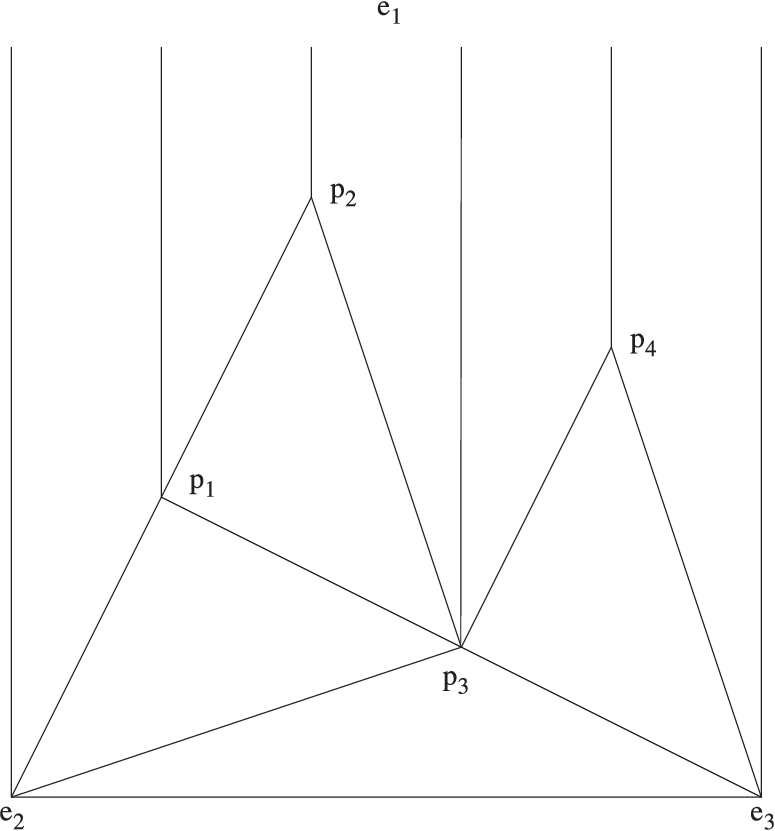}\\
  \caption{The fan of Danilov resolution of $\frac{1}{5}(1,2,3)$ cut with hyperplane $e_2^*+e_3^*=1.$}\label{figure:123 resolved}
\end{figure}

\begin{definition}
For fixed $r$ and $a,$ we call the resolution of ${1/(r-a)}(1,\modulo{r}{r-a},\modulo{-r}{r-a})$ singularity an $L$-resolution and the resolution of ${1/a}(1,\modulo{-r}{a},\modulo{r}{a})$ an $R$-resolution.
The $3$-dimensional cones of the fan of the Danilov resolution will be called $L$-cones or $R$-cones if they are subsets of the cones $\langle e_1,e_2,p_{r-a} \rangle$ or $\langle e_1,e_3,p_{r-a}\rangle,$ respectively.
\end{definition}

The fan of the Danilov resolution consist of $2r-1$ simplicial,
$3$-dimensional cones. Precisely $r$ cones of dimension $3$ contain $e_1.$ We call them $\sigma_0,\ldots,\sigma_{r-1}.$

\begin{definition}\label{d:cones sigma_i}
Set
\[ \sigma_i = \langle p_i,p_{i+1},e_1\rangle,\quad\text{for}\quad i=0,\ldots,r-1.\]
\end{definition}

Note that the resolution can be constructed by $r-1$ weighted blow-ups at the points
$p_1,\ldots,p_{r-1}$ with a suitable order.

\begin{corollary}\label{c:primitive resolutions}
The Danilov resolution of the singularity $\frac{1}{r}(1,1,r-1)$ is obtained
by the consecutive blow-ups at the points $p_{r-1},\ldots,p_{1}.$ In the case
of $\frac{1}{r}(1,r-1,1)$ the blow-ups are made at the points $p_1,\ldots p_{r-1}.$
\end{corollary}

\begin{definition}
Let $D_i$ denote the $T$-invariant toric divisor associated to the ray generated by the lattice point $p_i$ for $i=0,\ldots,r.$
Let $E_j$ denote the $T$-invariant toric divisor associated to the ray generated by $e_j$ for $j=1,2,3.$

\end{definition}
Note that $D_0=E_2$ and $D_r=E_3.$

\section{Divisors $X_i,Y_i,Z_i,R_i$ and their properties}\label{s:divisors X,Y,Z}

In this section we start by defining recursively a permutation $\tau,$ depending on $a,r.$ It will be
subsequently used in construction of divisors $X_i,Y_i,Z_i$ on the Danilov resolution.
These divisors will define the structure of a gnat family in the sense of Logvinenko~\cite{Logvinenko:RIMS} on the Danilov resolution.

%na ktorym stozku "zaczyna" sie Z-wiazka zwiazana z charakterem i w bazie 1,a,r-a

\begin{definition}\label{d:tau}
If $a\in\{1,r-1\}$ set $\tau(r,a,i)=\modulo{ai-1}{r}$ for $i=0,\ldots,r-1,$ and otherwise
\[
\tau(r,a,i)=\begin{cases}
\tau(r-a,\modulo{r}{r-a},\modulo{i}{r-a}), & \text{if $i\ge a$},\\
(r-a)+\tau(a,\modulo{-r}{a},i), & \text{if $i<a.$}\\
\end{cases}
\]
\end{definition}

Note that $\tau(r,a,0)=r-1.$ The function $\tau(r,a,\cdot)$ is a permutation of the set $\{0,\ldots,r-1\}.$ Note the recursive nature
of the above definition. The permutation $\tau$ will play a crucial role in determining the stability parameters connected with the moduli structure on the Danilov resolution.

\begin{example}\label{e:perm}
In order to compute $\tau(5,2)$ we need to know $\tau(2,1)=(0,1)$ and $\tau(3,2)=(0,2)(1)$ (written as cycles in the standard notation). Then,
\[\tau(5,2,i)=\tau(3,2,\modulo{i}{3}),\quad\ \text{for}\ i\ge 2,\]
\[\tau(5,2,i)=3+\tau(2,1,i),\quad\text{for}\ i<2,\]
 hence $\tau(5,2)=(0,4,1,3,2)$ as a cycle. In an analogous way we could check that $\tau(7,2)=(0,6,3,2)(1,5,4)$ which in turn allows us to compute $\tau(12,7)=(0,11,3,7)(1,10,4,6,8,2,5,9)$.
\end{example}

\begin{definition}
We call the sequence of numbers $i,i+(r-a),\ldots,i+s(r-a)$ an $L$-brick if
\begin{enumerate}[i)]
  \item $i<r-a,$
  \item $i+(s+1)(r-a)>r,$
  \item every number in the sequence is strictly smaller than $r.$
\end{enumerate}

%$i<r-a$ $i+(s+1)(r-a)>r$ and every number in the sequence is strictly smaller than $r.$
The sequence of numbers $i,i+a,\ldots,i+sa$ is called an $R$-brick if
\begin{enumerate}[1)]
  \item $i<a$
  \item $i+(s+1)a>r,$
  \item every number in the sequence is strictly smaller than $r.$
\end{enumerate}
%$i<a$ $i+(s+1)a>r$ and every number in the sequence is strictly smaller than $r.$
\end{definition}

The $L$- and $R$-bricks connect the
characters of a cyclic group of order $r$ with the characters of cyclic groups of order
$r-a$ and $a,$ respectively. They can be identified with fibers of the projections
$\ZZ/r\ZZ\too \ZZ/a\ZZ$ and $\ZZ/r\ZZ\too \ZZ/(r-a)\ZZ,$ which shows that there are $a$ different $R$-bricks and $r-a$ different $L$-bricks.

%Note that if $i,\ldots,i+(s+1)(r-a)$ is an $L$-brick then the coefficients of $Z_i,\ldots,Z_{i+s(r-a)}$ at
%$D_0, \ldots,D_{r-a}$ are equal to $0.$ On the other hand, if $i,\ldots,i+sa$ is an $R$-brick then
%the coefficients of $Y_{i+a},\ldots,Y_{i+sa}$ at $D_{r-a}\ldots,D_r$ are equal to $1.$ Observe also that %either
%all $L$-bricks or all $R$-bricks are of length at most $2.$

%Every $L$-brick contains a unique number (it ends with it, in fact) not smaller than $a,$ hence there are %exactly $r-a$ different $L$-bricks. Analogously, every $R$- brick starts with a unique number smaller than $a$ %so there are $a$ distinct $R$-bricks.

 %Summing up: any $i<a$ uniquely determines an $R$-brick and for any $R$-cone $\sigma$ the $z_i$-arrow is $\sigma$-distinguished if and only if the $z_i^R$-arrow is $\sigma$-distinguished. Moreover $z_i$-arrow
%is distinguished for $L$-cone $\sigma.$
%
%If $i\ge a$ then it uniquely determines an $L$-brick and the $z_i$-arrow is $\sigma$-distinguished
%if and only if the $z_\modulo{i}{r-a}$-arrow for $\frac{1}{r-a}$ is $\sigma$-distinguished for any $3$-dimensional $L$-cone $\sigma.$ Moreover $z_i$-arrow is not distinguished for any $R$-cone.
For fixed $r$ and $a$ let $Y(r,a)$ denote the Danilov resolution of the $\frac{1}{r}(1,a,r-a)$ singularity.
The rest of this section is devoted to finding effective toric divisors $X_i,Y_i,Z_i$ on $Y(r,a)$ for $i=0,\ldots, r-1.$ These divisors will yield $\QQ$-divisors $R_i,$  used in defining the structure of a moduli space on $Y(r,a).$
Note that the addition in the indices of $X_i,Y_i,Z_i$ is always meant modulo $r.$

\begin{definition}\label{d:divisors}
Let
\[Y_{i-a}=\sum_{k=0}^{\tau(r,a,i)} D_k,\quad\text{for }i=0,\ldots,r-1,\]
\[Z_{i}=\sum_{k=\tau(r,a,i)+1}^{r} D_k,\quad\text{for }i=0,\ldots,r-1,\]
and let the divisor $X_i$ be defined by the following equations:
\[X_{i}+Z_{i+1}=Z_{i}+X_{i-a},\quad\text{for }i=0,\ldots,r-1,\]
\[X_{0}=E_1.\]
\end{definition}

The last condition ensures that the divisors $X_i$ are uniquely determined. Note that by definition
\[Y_{i-a}+Z_i=\sum_{i=0}^r D_k,\]
that is $Y_{i-a}+Z_i$ does not depend on $i.$ Moreover
\[X_{i}+Y_{i+1}=Y_{i}+X_{i+a}.\]

The divisors $X_i,Y_i,Z_i$ satisfy following commutativity relations:
\begin{equation}\label{e:comm xy}
    X_{i}+Y_{i+1}=Y_{i}+X_{i+a},
\end{equation}
\begin{equation}\label{e:comm xz}
    X_{i}+Z_{i+1}=Z_{i}+X_{i-a},
\end{equation}
\begin{equation}\label{e:comm yz}
    Y_i+Z_{i+a}=Z_i+Y_{i-a}.
\end{equation}

\begin{lemma}\label{l:y_z_trivial}
For $a=1$ we have $Y_i=D_0+\ldots+D_i$ for $i=0,\ldots,r-1$ and $Z_0=D_r,\ Z_i=D_i+\ldots+D_r$ for $i=1,\ldots,r-1$. Moreover $X_i=Y_i+E_1-D_0$ for $i=0,\ldots,r-1$.
For $a=r-1$ we have $Y_0=D_0,\ Y_i=D_0+\ldots+D_{r-i}$ for $i=1,\ldots,r-1$ and $Z_i=D_{r-i}+\ldots+D_r,\ X_i=Z_i+E_1-D_r$ for $i=0,\ldots,r-1$.

\end{lemma}

\begin{proof}
To see this, it is enough to combine Definitions~\ref{d:tau} and~\ref{d:divisors}.
\end{proof}

\begin{example}\label{e:div_y_and_z}
For $r=5$ and $a=2$ we have $\tau(5,2)=(0,4,1,3,2)$ as a cycle, cf. Example~\ref{e:perm}, and hence
\begin{align*}
Y_0&=D_0, & Z_0&=D_5,\\
Y_1&=D_0+D_1+D_2, & Z_1&=D_4+D_5,\\
Y_2&=D_0+D_1, & Z_2&=D_1+D_2+D_3+D_4+D_5,\\
Y_3&=D_0+D_1+D_2+D_3+D_4, & Z_3&=D_3+D_4+D_5,\\
Y_4&=D_0+D_1+D_2+D_3, & Z_4&=D_2+D_3+D_4+D_5.
\end{align*}
By solving the linear equations $X_i+Z_{i+1}=Z_i+X_{i-2}$ in $X_i's$ we have:
\begin{align*}
X_0&=E_1,\\
X_1&=E_1+D_2+D_4,\\
X_2&=E_1+D_1+D_2,\\
X_3&=E_1+D_4,\\
X_4&=E_1+D_1+2D_2+D_2+D_3+D_4.
\end{align*}
\end{example}

\begin{definition}
    For fixed $a$ and $r,$ by $X^L_i,Y^L_i,Z^L_i$ we mean divisors on the $L$-resolution
    defined as in Definition~\ref{d:divisors} for $r_L=r-a$ and $a_L=\modulo{r}{r-a}.$ Similarly, by $X^R_i,Y^R_i,Z^R_i$ we mean divisors on the $R$-resolution defined for $r_R=a$ and $a_R=\modulo{-r}{a}.$ In particular
    \[Z^L_i=\sum_{k={\tau(r_L,a_L,i)+1}}^{r_L} D_k,\quad\text{for }i=0,\ldots,r_L-1,\]
    \[Z^R_i=\sum_{k={\tau(r_R,a_R,i)+1}}^{r_R} D_{k+r_L},\quad\text{for }i=0,\ldots,r_R-1.\]
    We note that in the definition of $Z^R_i$ there is a shift by $r_L$ in the index of $D_i$ since the divisor $D_{i+r_L}$ on the resolution $Y(r,a)$
    corresponds to the divisor $D_i$ on the resolution $Y(r_R,a_R)$.
\end{definition}

\begin{example}\label{e:L and R-divisors}
For $r=5$ and $a=2$, by Lemma~\ref{l:y_z_trivial}, we have
\begin{align*}
Z^L_0&=D_3,         \quad &X_0^L&=E_1,\\
Z^L_1&=D_2+D_3,     \quad &X_1^L&=E_1+D_2,\\
Z^L_2&=D_1+D_2+D_3, \quad &X_2^L&=E_1+D_1+D_2,\\
Z^R_0&=D_5,         \quad &X_0^R&=E_1,\\
Z^R_1&=D_4+D_5,     \quad &X_1^R&=E_1+D_4.
\end{align*}
Note the shift of indices in the divisors $D_i's$ on the $R$-resolution by $r_L=3$. Since the $L$-resolution
is a resolution of the $\frac{1}{3}(1,2,1)$ singularity, we have by Lemma~\ref{l:y_z_trivial} \[Y^L_0=D_0,\ \ Y^L_1=D_0+D_1,\ \ Y^L_2=D_0+D_1+D_2\]
On the other hand, the $R$-resolution is a resolution of the $\frac{1}{2}(1,1,1)$ singularity, hence
\[Y^R_0=D_3,\ \ Y^R_1=D_3+D_4.\]
\end{example}

Following propositions will prove useful in later sections.

\begin{lemma}\label{l:recursion for Z}
 Let $i,\ldots,i+s(r-a)$ be an $L$-brick. Restriction of the divisor $Z_{i+s(r-a)}$ to the $L$-resolution is equal to the divisor $Z^L_i.$ If $i,\ldots,i+sa$ is an $R$-brick then restriction of the divisor $Z_i$ to the $R$-resolution is equal to the divisor $Z^R_i.$
\end{lemma}

\begin{proof}
Observe that if $i,\ldots,i+s(r-a)$ is an $L$-brick, then $i+s(r-a)\ge a$ and $\modulo{i+s(r-a)}{r-a}=i.$ Therefore $\tau(r,a,i+s(r-a))=\tau(r-a,\modulo{r}{},i).$ If $i,\ldots,i+sa$ is an $R$-brick then $i<a$ and $\tau(r,a,i)=(r-a)+\tau(a,\modulo{-r}{},i).$
\end{proof}

Similar fact holds for restrictions of divisors $X_i.$

\begin{lemma}\label{l:recursion for X}
For any $i\le r-2$ the divisor $X_i$ restricted to the $L$-re\-so\-lu\-tion is equal to the divisor $X_{j}^L,$ where $j=\modulo{i}{r-a},$  and the divisor $X_i$ restricted to the $R$-re\-so\-lu\-tion is equal to the divisor $X_j^R,$ where $j=\modulo{i}{a}.$
%Moreover
%\[X_{r-1}=X_{r-a-1}+Z_{r-1}-Z_{0}.\]
\end{lemma}

\begin{proof}
We give the proof only in the case of restriction to the $R$-resolution.
%For every $i\ge a-1$ there exists a unique $R$-brick containing $i.$
First we show that if $i,\ldots,i+sa$ is an $R$-brick such that $i+sa\neq r-1,$ then the $R$-restrictions of the divisors $X_i,\ldots,X_{i+sa}$ are equal. To see this, observe that the restrictions of the divisors $Z_j$ for $j\ge a$
to the $R$-resolution are equal by definition  of the permutation $\tau$ and use the commutativity relation~\eqref{e:comm xz}
\[Z_j-Z_{j+1}=X_j-X_{j-a}\]
for $j=i+a,i+2a,\ldots,i+sa.$

If $i,\ldots,i+sa$ is an $R$-brick such that $i+sa=r-1$ then, by a similar proof, the restrictions of the divisors $X_i,\ldots,X_{i+(s-1)a}$ (i.e. all but the last) to the $R$-resolution are equal.

Hence, it is enough to prove the lemma assuming $i<a.$ Denote by $X_i|_R$ restriction of the divisor $X_i$ to the $R$-resolution. Obviously $X_0|_R=X_0^R$ and by Lemma~\ref{l:recursion for Z} we obtain relations
\[X_i|_R+Z_{i+1}^R=Z_i^R+X_{i-a}|_R,\]
for $i=0,\ldots a-2.$ We have proven already that $X_{i-a}|_R=%X_{i+r-a}=
X_{\modulo{i+r}{a}}|_R$ so the above relations can be rewritten as
 \[X_i|_R+Z_{i+1}^R=Z_i^R+X_\modulo{i+r}{a}|_R.\]
These are exactly the equations~\eqref{e:comm xz} for $r_R=a$ and $a_R=\modulo{-r}{a},$ so
\[X_i|_R=X_i^R\quad\text{for}\quad i=0,\ldots,a-2.\]

Let $j$ be the last element of an $R$-brick containing $a-1.$ Then $j\neq r-1$ so $0\le\modulo{j+a}{r}<a-1$ and the equation~\eqref{e:comm xz}
\[X_{j+a}+Z_{j+a+1}=X_j+Z_{j+a},\]
 restricted to the $R$-resolution %$Z_{j+a}|R=Z_{j+a}^R$ and $Z_{j+a+1}|_R=Z_{j+a+1}^R$
 becomes
 \[X_{j+a-r}^R+Z_{j+a-r+1}^R=X_{a-1}|_R+Z_{j+a-r}^R.\]
This finishes the proof as the above is exactly the equation~\eqref{e:comm xz} for $r_R=a,a_R=\modulo{-r}{a}$ and $i=j+a-r.$
\end{proof}

\begin{example}\label{e:restricitons}
In the case of $r=5$ and $a=2$ there are 3 different $L$-bricks $(0,3), (1,4),(2)$ and 2 different $R$-bricks $(0,2,4),(1,3).$
The $L$-resolution contains divisors $D_0,D_1,D_2,D_3,E_1$ and the $R$-resolution contains divisors $D_3,D_4,D_5,E_1.$ By Examples~\ref{e:div_y_and_z} and~\ref{e:L and R-divisors} we note that
\[Z_3|_L=Z_0^L,\quad
Z_4|_L=Z_1^L,\quad
Z_2|L=Z_2^L,\]
\[Z_0|_R=Z_0^R,\quad
Z_1|_R=Z_1^R.\]
Moreover,
\[X_0|_L=X_3^L=X_0^L,\ \ X_0|_R=X_2|_R=X_0^R,\]
\[X_1|_L=X_1^L,\ \ X_1|R=X_1^R.\]
\end{example}

\begin{lemma}\label{l:X_i eff}
For $a,r$ coprime, the divisors $X_i-E_1, Y_i-E_2, Z_i-E_3$  are effective for $i=0,\ldots,r-1.$
\end{lemma}
\begin{proof}
The result is true for $Y_i-E_2$ and $Z_i-E_3$ by definition, since $E_2=D_0$ and $E_3=D_r$. It remains to prove the result for $X_i-E_1$.

Note that for $a\in\{1,r-1\}$ either $X_i-E_1=Y_i-E_2$ or $X_i-E_1=Z_i-E_3$ by Lemma~\ref{l:y_z_trivial}.
For $1<a<r-1,$ by recursion and Lemma~\ref{l:recursion for X}, the restrictions
of $X_i-E_1$ to $L$- and $R$-resolution are effective for $i\neq r-1.$
Finally, note that
\[X_{r-1}-E_1=(X_{r-a-1}-E_1)+(Z_{r-1}-Z_0),\]
where both summands are effective, since $Z_0=D_r$ for any $r,a$.
\end{proof}

\begin{definition}\label{d:d_x}
Let
\[D_X=\frac{1}{r}\divisor(re_1^*),\]
\[D_Y=\frac{1}{r}\divisor(re_2^*),\]
\[D_Z=\frac{1}{r}\divisor(re_3^*),\]
be $\QQ$-divisors on $Y$ where $\divisor(re_i^*)$ denotes the
divisor of zeros and poles of the rational function $re_i^*$.
\end{definition}

We introduce the $\QQ$-divisors $R_i$  which later will define the desired family of McKay quiver representations on $Y(r,a).$

\begin{definition}\label{d:r_i}
For fixed $r$ and $a,$ define the $\QQ$-divisors $R_i$ for $i=0,\ldots,r-1$ by the equations
\[Z_i=D_Z+R_i-R_{i-a},\]
\[R_0=0.\]
\end{definition}

The divisors $R_i$ are uniquely determined by the condition $R_0=0$ since $r,a$ are coprime and the rank of the matrix determining equations for $R_i$ is equal to $r-1.$
Using the equation $Z_{i}+Y_{i-a}=D_Y+D_Z$ we get
\[Y_i=D_Y+R_i-R_{i+a}.\]

\begin{lemma}\label{l:auxiliary for R_i}
For any coprime $r$ and $a$
\[R_1=D_X-E_1.\]
\end{lemma}

\begin{proof}
By definition $R_1=(r-b)D_Z-(Z_0+Z_{-a}+Z_{-2a}+\ldots+Z_{-(r-b-1)a})=(r-b)D_Z-(Z_{a+1}+Z_{2a+1}+Z_{3a+1}+\ldots+Z_0)$ since $\modulo{-(r-b-i)a}{r}=ia+1$. Therefore it is enough to show that
\[E_1=D_X-(r-b)D_Z-(Z_{a+1}+Z_{2a+1}+\ldots+Z_{0}).\]
Assume that  $a\notin\{1,r-1\}$ as otherwise the statement is trivial, cf.~Lemma~\ref{l:y_z_trivial}. We use a recursive argument.
Observe that the numbers in the sequence
\begin{equation*}\tag{$\star$}\label{eq:sequence for R_i}
\modulo{a+1}{r},\modulo{2a+1}{r},\modulo{3a+1}{r},\ldots,\modulo{0}{r},
\end{equation*}
not greater than $a-1$ (i.e. the first numbers in $R$-bricks) are equal to the numbers \[\modulo{a_R+1}{a},\modulo{2a_R+1}{a},\modulo{3a_R+1}{a},\ldots,\modulo{0}{a},\]
where $a_R=\modulo{-r}{a}.$  Moreover, the numbers in the sequence~\eqref{eq:sequence for R_i}
%$\modulo{a+1}{r},\modulo{2a+1}{r},\modulo{3a+1}{r},\ldots,\modulo{0}{r},$
greater or equal to $a$ (i.e. the last numbers in $L$-bricks) are equal modulo $r-a$ to the numbers
\[\modulo{a_L+1}{r-a},\modulo{2a_L+1}{r-a},\modulo{3a_L+1}{r-a},\ldots,\modulo{0}{r-a},\]
where $a_L=\modulo{r}{r-a}.$ We omit a proof of this arithmetic fact. Let $W^L, W^R$ denote the pushforwards to the Danilov resolution of the
divisors $D^L_X-(r_L-b_L)D^L_Z$ and $D^R_X-(r_R-b_R)D^R_Z$ (where $a_L b_L=1$ modulo $r_L$ and $a_R b_R=1$ modulo $r_R$) computed in the fans of $L$- and $R$-resolutions in the lattices $N(r_L,a_L)$ and $N(r_R,a_R)$, respectively. Note that the first coordinate of the point $p_{i+1}$ is not smaller that the first coordinate of the point $p_i$ if and only if the toric ray dual to the cone $\langle p_i,p_{i+1}\rangle$ is equal to $e_1^*-(r-b)e_3^*,$ which gives an intrinsic explanation of the value of $e_1^*-(r-b)e_3^*$ on the generator of a ray. Since $(e_1^*-(r-b)e_3^*)(\frac{1}{r}(1,a,r-a))=b_L-r_L$, we get $D_X-(r-b)D_Z=W^L+(b_L-r_L)W^R-E_1$, as $E_1$
is counted twice. To finish observe that in the sequence ($\star$) exactly $r_L-b_L$ numbers are greater or equal to $a$ and use Lemma~\ref{l:recursion for Z}.
\end{proof}

%\begin{lemma}
%The $\QQ$-divisors $R_i$ satisfy
%\[X_i=D_X+R_i-R_{i+1}\quad\text{for}\quad i=0,
%\ldots r-1.\]
%\end{lemma}
\begin{lemma}\label{c:reductor condition}
    The divisors $R_i$ satisfy the following equations for $i=0,\ldots,r-1.$
    \[X_i=D_X+R_i-R_{i+1},\]
    \[Y_i=D_Y+R_i-R_{i+a},\]
    \[Z_i=D_Z+R_i-R_{i-a}.\]
\end{lemma}

\begin{proof}
It is enough to prove the first equality. Set $\widetilde{X_i}=D_X+R_i-R_{i+1}$ and note that $\sum \widetilde{X_i}=rD_X.$ Moreover
the divisors $\widetilde{X_i}$ satisfy commutativity relations~\eqref{e:comm xz}, hence
$X_i-\widetilde{X_i}$ is constant. Since $\widetilde{X_0}=D_X-R_1=X_0,$ by Lemma~\ref{l:auxiliary for R_i}, the constant is equal to 0.
\end{proof}

\section{The McKay quiver}\label{section:McKay quiver}

By a quiver we mean a finite, directed graph $Q.$ The set of vertices of $Q$
will be denoted by $Q_0$ and the set of arrows by $Q_1.$ For any arrow $a$ in $Q_1$ denote by $\tl(a)$
 the tail of $a$ and by $\hd(a)$ denote the head of $a.$ In what follows, we restrict the general definitions
 of quiver representations to the simple case where the dimension vector is equal to $(1,\ldots,1).$ For any $v\in Q_0$ let $\CC_v$
 denote $1$-dimensional complex vector space assigned to the vertex $v.$
 Representation of the quiver $Q$  is an element of
\[\Rep(Q)=\bigoplus_{a\in Q_1} \Hom_\CC(\CC_{\tl(a)},\CC_{\hd(a)}).\]
By fixing a basis in each $\CC_v$ we can identify $\Rep(Q)$ with
an affine space. With this choice, for any representation $V\in\Rep(Q)$ and $a\in Q_1$ denote by $V(a)$ the constant
representing arrow $a$ in $V.$

A path $q$ in quiver $Q$ is a sequence of arrows $a_l,\ldots, a_2,a_1$ where $\hd(a_i)=\tl(a_{i+1}).$
A linear combination of paths $q_i$ is called an admissible relation, if paths $q_i$ have
the same heads and tails. Any set $R$ of admissible relations
for quiver $Q$ defines an affine subscheme of $\Rep(Q)$ cut by the polynomial
equations coming from $R,$ i.e.
\[\Rep(Q,R):=\{V\in \Rep(Q)\; ; \; V(c)=0 \quad\mbox{for}\ c\in R\},\]
where the function $V(c)$ denotes the linear extension of the function $V(q)=V(a_l)\cdot \ldots \cdot V(a_1)$
defined for a path $q$ consisting of arrows $a_l,\ldots ,a_1.$

Two representations of quiver $Q$ are isomorphic if and only if they lie in the same orbit of the group
$\GL(Q,\CC)=\bigoplus_{v\in Q_0}\CC^*,$
acting on the left on the set $\Rep(Q)$ in the following way:
\[(g\cdot V)(a)=g(\hd(a))V(a)g(\tl(a))^{-1},\quad\text{for\ any}\quad V\in\Rep(Q).\]
This action leaves $\Rep(Q,R)$ invariant. Dividing by the $1$-dimensional subgroup acting trivially we are left with a faithful action of the group
\[\PGL(Q,\CC)=\GL(Q,\CC)/\CC^*(1,\ldots,1).\]

\begin{definition}\label{d:subreps with nonzero arrows}
By a subquiver $Q'\subset Q$ we mean a subset of vertices $Q'_0\subset Q_0$ and a subset of arrows $Q'_1\subset Q_1$
satisfying the following condition: $\tl(a),\hd(a)\in Q'_0$ for any $a\in Q'_1.$ Let $V\in\Rep (Q,R)$ be a representation of $Q.$
A subrepresentation $V'$ of representation $V$ is a representation of a subquiver $Q'\subset Q$
satisfying the following conditions:
$$V'(a)=V(a)\mbox{ for any }a\in Q'_1,$$
$$\mbox{if}\ \ \tl(a)\in Q'_0,\; V'(a)\neq 0\ \ \mbox{then}\ \ \hd(a)\in Q'_0, a\in Q'_1,\mbox{ for any }a\in Q_1.$$
\end{definition}

We do not need the general definition of the McKay quiver, so we quote only the specialization to the case
of a cyclic group action.

\begin{definition}(McKay)\label{definition:McKay quiver}
Let $G$ be a cyclic group $G\subset\GL(3,\CC)$ of order $r,$ such that
the quotient singularity $\CC^3 / G$ is of type
$\frac{1}{r}(1,a,r-a).$ Define the McKay quiver for group $G$ as a finite graph
with $r$ vertices $0,1,\ldots,r-1$  and $3r$ arrows $x_0,y_0,z_0,\ldots,x_{r-1},y_{r-1},z_{r-1}$ such
that $\tl(x_i)=\tl(y_i)=\tl(z_i)=i$ and
\[\hd(x_i)=\modulo{i+1}{r},\quad
\hd(y_i)=\modulo{i+a}{r},\quad
\hd(z_i)=\modulo{i-a}{r}.\]
\end{definition}

The vertices of the McKay quiver correspond to the characters of $G.$

\begin{figure}[h] \centering \includegraphics[width=0.5\textwidth]{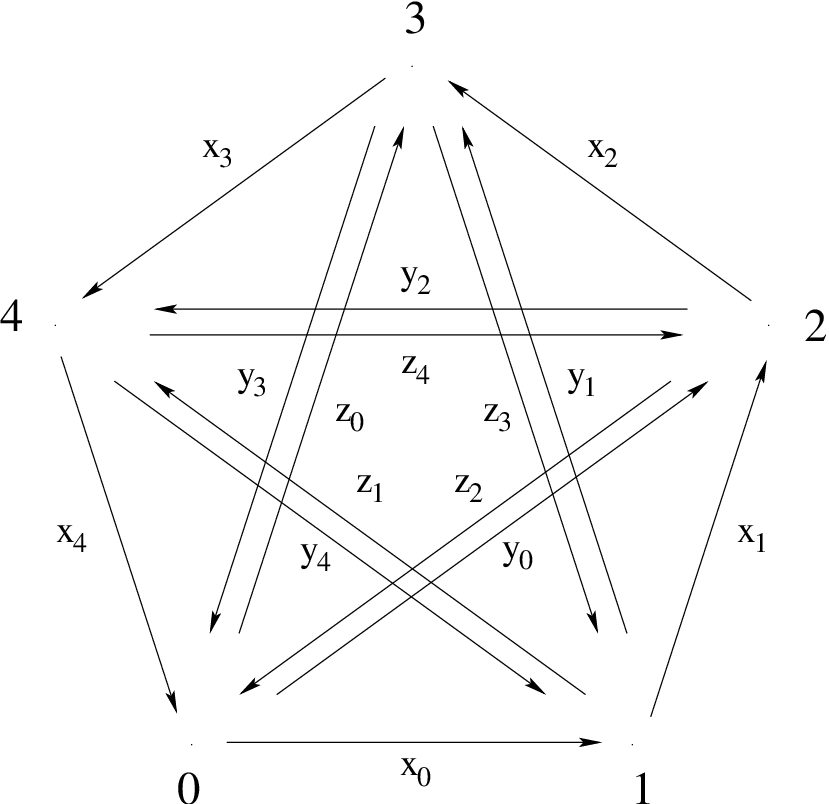}
\caption{The McKay quiver for $r=5,a=2.$} \label{figure:McKay quiver}
\end{figure}

\begin{definition}\label{definition:rep of McKay quiver}
A representation of the McKay quiver is an element of \\
$\Rep(Q,R),$ where $Q$ is the McKay quiver (for fixed $r,a$) and $R$ is the set relations
\[R=\{y_{i+1}x_i-x_{i+a}y_i,z_{i+1}x_i-x_{i-a}z_i,y_{i-a}z_i-z_{i+a}y_i\;
; \; i=0,\ldots,r-1 \},\]
where all indices are meant modulo $r.$
\end{definition}

\section{Family of representations of the McKay quiver}\label{section:construction of family}

In this section we will define a family of the McKay quiver
representations over the Danilov resolution using line bundles determined by the effective divisors $X_i, Y_i, Z_i.$

\begin{definition}[King, Logvinenko]\label{d:KingLogv}
Fix coprime $r,a$ and let $Y(r,a)$ be the Danilov resolution. By a
family of McKay quiver representations on $Y(r,a)$ for the action of type $\frac{1}{r}(1,a,r-a)$ we mean a
set of $r$ $\QQ$-divisors $R_i,$ for $i=0,\ldots,r-1,$ on $Y(r,a)$, such that the $\QQ$-divisors
 \[X_i=D_X+R_i-R_{i+1},\]
 \[Y_i=D_Y+R_i-R_{i+a},\]
 \[Z_i=D_Z+R_i-R_{i-a}.\]
 are effective divisors for $i=0,\ldots,r-1$ (where $D_X, D_Y, D_Z$ are as in Definition~\ref{d:d_x}).
\end{definition}

The above condition is called {\it the reductor condition} in~\cite{Logvinenko:RIMS}. It ensures that the family
of quiver representations on the resolution $Y(r,a)$ is {\it natural}, that is the support of a representation (seen
as a finite dimensional $\CC[x,y,z]$-module) parameterized by a point $p$
in the resolution  coincides with the $G$-orbit parameterized by the image of $p$ in the quotient space, cf.~\cite[Definition~1.4]{Logvinenko:RIMS}.
%The divisors $X_i,Y_i,Z_i$ defined in Section~\ref{section:construction of family}

\begin{definition}\label{d:family F}
For fixed $a$ and $r$ denote by $\FFF(r,a)$ (or $\FFF$ for short) the family given by the $\QQ$-divisors $R_i$ from Definition~\ref{d:r_i}.
\end{definition}

\begin{remark}Note that the divisors $R_i$ satisfy {\it the reductor condition}  by Lemma~\ref{c:reductor condition} and the divisors $X_i,Y_i,Z_i$ are effective by Lemma~\ref{l:X_i eff}.\end{remark}

One could define the family $\FFF(r,a)$ by representing the arrows of the McKay quiver by the sections from Definition~\ref{d:KingLogv}. Equivalently,
define line bundles $L_i=\OO_Y(-R_i)$ on $Y$ for $0\le i \le r-1.$ Then, for each $0 \le i \le r-1$, multiplication by the sections defining the
divisors $X_i, Y_i$ and $Z_i$, respectively, determines morphisms from $L_i$ to $L_{i+1}$, to $L_{i+a}$ and to $L_{i-a}$, respectively.
%Note that we could take an equivalent approach to defining the family $\FFF(r,a)$ by representing the arrows of the McKay quiver by the natural sections of the line bundles $\OO(X_i),\OO(Y_i),\OO(Z_i).$ However, since we need to use Logvinenko's result from~\cite{Logvinenko:RIMS} we comply with his notation. It is also easier, for the purpose of this paper, to deal with explicit representations of the arrows.

We will show later that there exist stability conditions $\theta$ such that every representation in the family $\FFF(r,a)$ is $\theta$-(semi)stable.
In fact, it will turn out that such stability conditions $\theta$ are exactly those for which the representations parameterized by $T$-fixed point of the cones $\sigma_0,\ldots,\sigma_{r-1}$ are
simultaneously $\theta$-(semi)stable.

\begin{definition}
For fixed family $\FFF$ and for any $3$-dimensional cone $\sigma$ in the fan of the Danilov resolution, we call an arrow in the McKay
quiver $\sigma$-distinguished if the corresponding divisor (that is $X_i, Y_i$ or $Z_i$)
does not contain the $T$-fixed point of toric chart $U_\sigma.$
\end{definition}

Observe, for example, that no $z_i$-arrow is $\sigma_{r-1}$-distinguished since the divisors $Z_i-E_3$ are
effective for any $i.$

\begin{lemma}\label{l:connectedness}
For any $3$-dimensional cone $\sigma$ in the fan of the Danilov resolution any two vertices of the McKay quiver can be connected by an undirected path of $\sigma$-distinguished arrows different from $x_{r-1}.$
\end{lemma}

\begin{proof}
The lemma is true for $a\in\{1,r-1\}.$ Note that any two vertices of the McKay quiver
lying in the same $L$-brick can be joined by a sequence of $z$-arrows for any $L$-cone $\sigma,$ and any two vertices lying in the same $R$-brick can be joined by a sequence of $y$-arrows if $\sigma$ is an $R$-cone.
By the inductive step, any two bricks can be joined by a sequence of $\sigma$-distinguished arrows.
To finish, it is enough to consider the cone $\sigma=\langle p_0,p_{r-a},p_r\rangle$ and observe that the only $\sigma$-distinguished arrows are $x_0,\ldots,x_{r-2}.$
\end{proof}

\begin{lemma}
Let $p,p'\in Y$ be two points in the Danilov resolution, belonging to two distinct toric charts isomorphic
to $\CC^3.$ Then the representations parameterized by $p$ and $p'$ in $\FFF$ are not isomorphic.
\end{lemma}

\begin{proof}
Let $\sigma$ and $\sigma'$ be $3$-dimensional cones in the fan of the Danilov resolution corresponding to charts containing $p$ and $p'.$ There are at most two common primitive generators of the cones $\sigma$ and $\sigma'$ belonging to the set $\{p_0,\ldots,p_r\}.$ This implies that at least one $y$- or $z$-arrow is $\sigma$-distinguished and not $\sigma'$-distinguished. Hence the representations parameterized by $p$ and $p'$ in $\FFF$ are not isomorphic.
\end{proof}

\begin{lemma}
Let $p,p'\in Y$ be distinct points in the Danilov resolution, belonging to a single toric chart, isomorphic to $\CC^3,$ on the Danilov resolution. The representations parameterized by $p$ and $p'$ in $\FFF$ are not isomorphic.
\end{lemma}

\begin{proof}
Let $\sigma=\langle p_l,p_m,p_n \rangle$ be the $3$-dimensional cone in the fan of Danilov resolution, such that $p,p'\in U_\sigma,$ where
$U_\sigma$ stands for the toric chart given by $\sigma.$
Let $D^\sigma_l,D^\sigma_m,D^\sigma_n$ be restrictions of the divisors $D_l,D_m,D_n$ to the chart $U_\sigma.$
We claim that there exist $i,j,k\in\{0,\ldots,r-1\},i\neq r-1$ such that the restrictions of $X_{i},Y_{j},Z_{k}$ to the chart $U_\sigma$ are equal to $D^\sigma_l,D^\sigma_m,D^\sigma_n,$ respectively. This holds for $a\in\{1,r-1\}$ and can be proven for $a\notin\{1,r-1\}$ using recursion and the Lemmata~\ref{l:recursion for Z},~\ref{l:recursion for X}.
In the orbit of the group $\GL(Q,\CC)$ there exists exactly one representation, such that all
$\sigma$-distinguished arrows are represented by the number $1$ (by Lemma~\ref{l:connectedness}).
Therefore, in this unique element of the orbit, the arrows $x_{i},y_{j},z_{k}$ are represented by toric coordinates on $U_\sigma.$ The points $p,p'$ have at least one toric coordinate different, therefore they parameterize non-isomorphic representations.
\end{proof}

\begin{corollary}\label{corollary:all reps non-iso}
For any two distinct points $p,p'\in Y$ in the Danilov resolution the representations parameterized by $p$ and $p'$ in $\FFF$ are not isomorphic.
\end{corollary}

\begin{definition}\label{d:FFF_i}
Let $\FFF_i$ denote the representation of the McKay quiver parameterized by the unique $T$-fixed point belonging to the toric chart $U_{\sigma_i}$ (cf.~Definition~\ref{d:cones sigma_i})
\end{definition}

 Since the divisors $X_j-E_1$ are effective no $x_j$-arrow is $\sigma_i$-distinguished for any $i.$ Moreover, by definition of permutation $\tau,$ for any $i$ there exists a unique $j,$ such that $y$-arrow and $z$-arrow joining vertices $j,j+a$ are not $\sigma_i$-distinguished. For $i$ and $j$ as above, if $j'\neq j$ then among the $y$- and  $z$-arrows joining $j',j'+a$  exactly one is $\sigma_i$-distinguished (cf.~Example~\ref{e:reps F_i}). Hence the representations $\FFF_i$ are particularly easy to deal with.

\section{Stability of quiver representations}\label{section:stability}

In this section we recall some facts and definitions concerning
$\theta$-stability of quiver representations (see~\cite{King}, note that we restrict to the case of dimension vector equal to $(1,\ldots,1)$). We prove that
that the representations in family $\FFF$ (cf.~Definition~\ref{d:family F})  on the Danilov resolution are simultaneously $\theta$-(semi)stable if and only if the representation $\FFF_0,\ldots,\FFF_{r-1}$ are $\theta$-(semi)stable.

For any quiver $Q$ set
$$\Wt(Q)=\{\theta:{Q_0\too \QQ}\;\; ; \; \sum_{v\in Q_0} \theta(v)=0\}.$$

Given a function $\theta\in\Wt(Q)$ for which $\theta(Q_0)\subset\ZZ$, we obtain a character $\chi_\theta$ of $\PGL(Q,\CC).$
Explicitly, for any such $\theta$ the character is
$$\chi_\theta(g)=\prod_{v\in Q_0} g(v)^{\theta(v)},$$
where $g\in \PGL(Q,\CC).$
Therefore, we will call $\Wt(Q)$ a weight space for $Q.$

\begin{definition}(A. King)\label{definition:def of theta-ss}
For any subrepresentation $V'$ of a representation $V\in \Rep(Q)$ and $\theta\in \Wt(Q)$
set
\[\theta(V')=\sum_{v\in Q'_0} \theta(v) .\]
Representation $V$ is called $\theta$-semistable if for every
proper, non-zero subrepresentation $V'\subset V,$
\[\theta(V') \ge \theta(V)=0.\]
Representation $V$ is called $\ths$ if an analogous condition with
strict inequality holds.
\end{definition}

\begin{definition}
We say that a stability parameter $\theta\in\Wt(Q)$ is generic if every $\theta$-semistable representation
is $\theta$-stable.
\end{definition}

\begin{theorem}[King]
Let $Q$ be a quiver, $R$ a set of admissible relations for $Q$ and let
$\theta\in\Wt(Q).$ A point in $\Rep(Q,R)$ is
$\chi_\theta$-(semi)stable under the action of $\PGL(Q,\CC)$ if
and only if the corresponding representation of $Q$ is
$\theta$-(semi)stable. Denote by $M_\theta(Q,R)$ the GIT quotient of
$\Rep(Q,R)$ by $\PGL(Q,\CC)$ with respect to the
$\chi_\theta$-linearization of the trivial bundle over $\Rep(Q,R).$
That is, $M_\theta(Q,R)$ is a scheme constructed from the
graded ring of semi-invariants, namely:
$$M_\theta(Q,R):=\Rep(Q,R){//}_{\chi_\theta} \PGL(Q,\CC)=\Proj \bigoplus_{k=0}^\infty
\CC[\Rep(Q,R)]^{\chi_\theta^k},$$ where
elements of the set $\CC[\Rep(Q,R)]^{\chi}$ are regular
functions $f$ on the representations space $\Rep(Q,R),$ such
that $f(g\cdot v)=\chi(g) f(v),$ for any $g\in \PGL(Q,\CC)$ and
any $v\in \Rep(Q,R).$

For a generic $\theta$ the variety $M_\theta(Q,R)$ is a
fine moduli space of $\theta$-stable representations.
\end{theorem}

\begin{proof}
See~\cite[Propositions~3.1,\ 5.2,\ 5.3]{King}.
\end{proof}

We need a fact concerning families of $\theta$-(semi)stable quiver representations on affine toric varieties, which
is true in more general setting.

\begin{lemma}\label{l:spread}
Let $U_\sigma$ be an affine toric chart in the Danilov resolution $Y$ containing a unique $T$-fixed point $p_\sigma.$
Let $\FFF$ be a family of McKay quiver representations on $Y$ as in Definition~\ref{d:KingLogv}.
If the representation parameterized by the point $p_\sigma$ is $\theta$-(semi)stable than all representation in $\FFF$ are $\theta$-(semi)stable.
\end{lemma}

\begin{proof}
The $\theta$-(semi)stability is an
open condition and it is invariant under the $T$-action since the divisors $R_i$ are $T$-equivariant. Moreover, the $T$-fixed point lies in the closure of all orbits in $U_\sigma.$
\end{proof}

We need also two simple lemmas.

\begin{lemma}\label{l:obvious}
Let $U_\sigma$ be a toric chart in the Danilov resolution, where $\sigma=\langle p_l,p_m,p_n\rangle$ and $l<m<n.$
If the arrow $y_i$ is $\sigma$-distinguished
then it is $\sigma_j$-distinguished for $j\ge l.$ If the arrow $y_i$ is not $\sigma$-distinguished then it is not $\sigma_j$-distinguished for $j\le l$.
If the arrow $z_i$ is $\sigma$-distinguished
then it is $\sigma_j$-distinguished for $j<n.$ If the arrow $z_i$ is not $\sigma$-distinguished then it is not $\sigma_j$-distinguished for $j\ge n-1$.
 \end{lemma}

\begin{proof}
It follows directly from the Definition~\ref{d:divisors}.
\end{proof}

\begin{lemma}\label{l:L_bricks_are_z-connected}
Let $i,i+a,\ldots,i+sa$ be an $R$-brick. Then for $j=i,i+(r-a),\ldots,i+(s-1)(r-a)$ the $y_j$-arrows are $\sigma$-distinguished for any cone $\sigma$ in $R$-resolution. Let $i,i+(r-a),\ldots,i+s(r-a)$ be an $L$-brick. Then for $j=i+a,\ldots,i+s(r-a)$ the $z_j$-arrows are $\sigma$-distinguished for any cone $\sigma$ in $L$-resolution.
\end{lemma}

\begin{proof}
Let $i,i+(r-a),\ldots,i+s(r-a)$ be an $L$-brick. Then $i+k(r-a)<a$ for $k=0,\ldots,s-1$. Hence, by Definition~\ref{d:tau}, $\tau(r,a,i+k(r-a))\ge r-a.$
This implies that the support of divisors $Z_i,Z_{i+(r-a)},\ldots,Z_{i+s(r-a)}$ is disjoint from any torus-fixed point on $L$-resolution. The other case
follows analogously.
\end{proof}

Using the above tools we can prove that $\theta$-stability of the representations $\FFF_1,\ldots,\FFF_{r-1}$ (see Definition~\ref{d:FFF_i}) controls stability
of the whole family $\FFF.$

\begin{lemma}
Let $\theta$ be a stability parameter such that $\FFF_0,\ldots,\FFF_{r-1}$ are $\theta$-(semi)stable. Then every representation in the family $\FFF$ is $\theta$-(semi)stable.
\end{lemma}

\begin{proof}
By Lemma~\ref{l:spread}, to conclude, it is enough to prove that every representation in $\FFF$, parameterized by a $T$-fixed point is $\theta$-(semi)stable.

Let $p_\sigma\in U_\sigma$ be a $T$-fixed point, where $\sigma=\langle p_l,p_m,p_n\rangle$ is a $3$-dimensional cone in the fan of the Danilov resolution and $l<m<n.$ Let $V$ be a subrepresentation of the representation parameterized by $p_\sigma$ in the family $\FFF$. The proof is by induction on $r.$ We will show that there exists $j\in\{0,\ldots,r-1\}$ and a subrepresentation of $\FFF_j$ supported on the same set of vertices as $V.$

The theorem is trivial to check if $a\in\{1,r-a\}.$ Assume that $r>1$ and the theorem is true for any $r'<r.$
Let $V$ be a subrepresentation as above. By $S(V)\subset\{0,\ldots,r-1\}$ we mean a subset of the vertices of the McKay quiver, supporting $V.$  Consider a sequence $i,i+(r-a),\ldots,i+s(r-a)$ of vertices in the set $S(V),$ such that the vertices $i-(r-a),i+(s+1)(r-a)$ are not in $S(V).$ The set $S(V)$ is a union of such sequences. There is no loss of generality in assuming that $S(V)$ itself is a single sequence. Note that $y_i$-arrow and $z_{i+s(r-a)}$-arrow are not $\sigma$-distinguished ($V$ is a subrepresentation, see Definition~\ref{d:subreps with nonzero arrows}).

Suppose that $z_{i-(r-a)}$-ar\-row is $\sigma$-distinguished or the $y_{i+(s+1)(r-a)}$-ar\-row is $\sigma$-dis\-tin\-guished.  We can assume that $k<r+1,$ otherwise there is nothing to prove. By Lemma~\ref{l:obvious}, the vertices $i,\ldots,i+s(r-a)$ form a subrepresentation of some of the representations $\FFF_{i},\ldots,\FFF_{k-1}.$
% say if $z_{i-(r-a)}$-ar\-row is $\sigma$-distinguished, from the fact that $z_{i-(r-a)}$-ar\-row is $\sigma_l$-distinguished  and the %$z_{i+s(r-a)}$- arrow is not $\sigma_l$-distinguished, for $l=i,\ldots, j-1.$

Now we turn to the case when both $z_{i-(r-a)}$-ar\-row and $y_{i+(s+1)(r-a)}$-ar\-row are not $\sigma$-distinguished. Assume that $\sigma$ is an  $L$-cone. Since $y_i$-arrow and $z_{i+s(r-a)}$ arrow are not $\sigma$-distinguished, the sequence $i,\ldots,i+s(r-a)$ is concatenated out of some $L$-bricks, by Lemma~\ref{l:L_bricks_are_z-connected}. These $L$-bricks correspond to vertices of the McKay quiver for ${1}/{(r-a)}(1,\modulo{r}{},\modulo{-r}{}).$ Moreover, the vertices corresponding to these $L$-bricks form a subrepresentation of the representation parameterized by $p_\sigma$ in the family $\FFF(r-a,\modulo{r}{})$ on the $L$-resolution.  Now we can use the inductive assumption.
\end{proof}

\section{Stability of the representations $\FFF_0,\ldots,\FFF_{r-1}$}\label{section:some linear algebra}

We proved that every representation in the family $\FFF$ family is $\theta$-(semi)stable if and only if
the representations $\FFF_0,\ldots,\FFF_{r-1}$ are simultaneously $\theta$-(semi)stable. We will show how
 to get such parameters $\theta$ using permutation $\tau.$

\begin{definition}
   Let $\xi(r,a)=\tau(r,a)^{-1}$ denote the inverse of permutation $\tau$ (see Definition~\ref{d:tau}).
\end{definition}

Since no $x_i$-arrow is $\sigma_j$-distinguished any two vertices of $\FFF_j$
can be joined by a sequence of $z$- and $y$-arrows, by Lemma~\ref{l:connectedness}.  Moreover, the arrows $z_{\xi(r,a,j)}$ and $y_{\xi(r,a,j)+(r-a)}$ are not $\sigma_j$-distinguished. Therefore, the quiver
 supporting representation $\FFF_j$ consists of vertices $0,1,\ldots,r-1$ and every two vertices $i,i+(r-a)$ are joined either by $z$-arrow or $y$-arrow (but not both) unless $i=\xi(r,a,j).$

\begin{definition}
Let
\[\overline{V}_j=\{ i\; : \; \text{vertices }i,i+(r-a),\ldots,\xi(r,a,j)\text{ of }Q\text{ form a subrepresentation of }\FFF_j\},\]
\[\overline{W}_j=\{ i\; : \; \text{vertices }\xi(r,a,j)+(r-a),\ldots, i\text{ of }Q\text{ form a subrepresentation of }\FFF_j\}.\]
For any $i\in\overline{V}_j$ let $V_{i,j}$ be the subrepresentation of $\FFF_j$ consisting of vertices $i,i+(r-a),\ldots,\xi(r,a,j).$
For any $i\in\overline{W}_j$ let $V_{i,j}$ be the subrepresentation of $\FFF_j$ consisting of vertices $\xi(r,a,j)+(r-a),\ldots,,i-(r-a), i.$
\end{definition}

Note that $i\in \overline{V}_j$ if and only if the $y_i$-arrow is not $\sigma_j$-distinguished.
Note that $i\in \overline{W}_j$ if and only if the $z_i$-arrow is not $\sigma_j$-distinguished.

\begin{lemma}\label{l:F_and_V,W}
Let $\theta\in\Wt(Q)$ be a fixed stability parameter. The representation $\FFF_j$ is $\theta$-semistable
if and only if $\theta(V_{i,j})\ge 0$ for any $i\in \overline{V}_j$ and $\theta(W_{i,j})\ge 0$ for any $i\in \overline{W}_j$.
It is $\theta$-stable if and only if the above conditions hold with strict inequalities.
\end{lemma}

\begin{proof}
The `only if' direction is obvious. Let $U$ be a subrepresentation of $\FFF_j.$ Without loss of generality assume it is supported on vertices
$i,i+(r-a),\ldots,i+s(r-a).$ Then, by Definition~\ref{d:subreps with nonzero arrows},
$i\in\overline{V}_j$ and $i+s(r-a)\in\overline{W}_j$ and hence $V_{i,j}, W_{i+s(r-a),j}$ are subrepresentations of $\FFF_j.$
The lemma follows since $\theta(U)=\theta(V_{i,j})+\theta(W_{i+s(r-a),j}).$
\end{proof}

\begin{example}\label{e:reps F_i}
The following diagram shows representations $\FFF_0,\ldots,\FFF_4$, respectively, in the case of $\frac{1}{5}(1,2,3)$ (c.f. Example~\ref{e:div_y_and_z}). The solid arrows
stand for arrows represented by a non-zero number i.e. $\sigma_j$-distinguished.
\[\FFF_0:\quad 0\overset{z_o}\too 3\overset{z_3}\too 1\overset{z_1}\too 4\overset{z_4}\too 2\]
\[\FFF_1:\quad 2\overset{y_0}\oot 0\overset{z_0}\too 3\overset{z_3}\too 1\overset{z_1}\too 4\]
\[\FFF_2:\quad 1\overset{z_1}\too 4\overset{y_2}\oot 2\overset{y_0}\oot 0\overset{z_0}\too 3\]
\[\FFF_3:\quad 4\overset{y_2}\oot 2\overset{y_0}\oot 0\overset{z_0}\too 3\overset{y_1}\oot 1\]
\[\FFF_4:\quad 3\overset{y_1}\oot 1\overset{y_4}\oot 4\overset{y_2}\oot 2\overset{y_0}\oot 0\]

Moreover $\xi(5,2,0)=2,\ \xi(5,2,1)=4,\ \xi(5,2,2)=3,\ \xi(5,2,3)=1,\  \xi(5,2,4)=0$
and for example $\overline{V}_2=\{1,3,4\},\ \overline{W}_2=\{2,3,4\}.$ If $U$ is a subrepresentation of $\FFF_2$
consisting of vertices $2,4$ then $\theta(U)=\theta(V_{4,2})+\theta(W_{2,2}).$

\end{example}

\begin{definition}
Let $\varphi(r,a,j)=\modulo{\xi(r,a,j)+(r-a)}{r}$ for any coprime $a,r$ and $j=0\ldots,r-1.$
\end{definition}

Note that since $\xi(r,a)$ is a permutation of the set $\{0,\ldots,r-1\}$ so is $\varphi(r,a).$ Let $n_0,\ldots,n_{r-1}$ be some rational numbers. The addition in the indices of $n_i's$ is modulo $r$. Observe that there is a linear map $\QQ^r\ni(n_0,\ldots,n_{r-1})\too\Wt(Q)$ of rank $r-1$ sending $(n_0,\ldots,n_{r-1})$ to the function $\theta\in\Wt(Q)$ such that $\theta(i)=n_i-n_{i+(r-a)}$ for $i=0,\ldots,r-1$ with kernel spanned by the vector $(1,\ldots,1)$.

\begin{lemma}\label{l:common stab}
For any rational numbers $n_0,\ldots,n_{r-1}$ set $\theta(i)=n_i-n_{i+(r-a)}$.
The representations $\FFF_0,\ldots,\FFF_{r-1}$ are simultaneously $\theta$-stable if and only if
\[n_{\varphi(r,a,0)}< n_{\varphi(r,a,1)}<\ldots < n_{\varphi(r,a,r-1)}.\]
\end{lemma}

\begin{proof}
Fix $j$ and $i\neq\xi(r,a,j)+(r-a).$ Then either the $z_{i-(r-a)}$- or the $y_i$-arrow is $\sigma_j$-distinguished. In the first case, $i\in\overline{V}_j$  and $\theta(V_{i,j})=n_i-n_{\xi(r,a,j)+(r-a)}> 0.$ Otherwise  $i-(r-a)\in\overline{W}_j$ and $\theta(W_{i-(r-a),j})=-n_i+n_{\xi(r,a,j)+(r-a)}> 0.$ By the definition of the permutation $\tau$ and by the definition of the divisors $Y_i,Z_i,$ exactly $r-1-j$ of $z$-arrows are $\sigma_j$-distinguished. Moreover, if $y_{j'}$-arrow
is $\sigma_{j'}$-distinguished then it is $\sigma_j$-distinguished for any $j\ge j'$, c.f. Lemma~\ref{l:obvious}. To prove the `only if' direction note that if $\FFF_0$ is $\theta$-stable then $n_{\varphi(r,a,0)}$ is the smallest of $n_i's.$ If in addition $\FFF_1$ is $\theta$-stable then $n_{\varphi(r,a,1)}$ is the second smallest, and so on. For the other direction note that for any $j=0,\ldots,r-1$ we have $\theta(V_{i,j})>0$ for any $i\in\overline{V}_j$ and $\theta(W_{i,j})>0$ for any $i\in\overline{W}_j.$
By Lemma~\ref{l:F_and_V,W} the representation $\FFF_j$ is $\theta$-stable for $j=0,\ldots,r-1.$
\end{proof}

Observe that for any $r,a$ coprime the set of conditions for which the representations $\FFF_0,\ldots,\FFF_{r-1}$ are simultaneously $\theta$-stable
is a simplicial cone (in particular it is non empty) since it is in bijection with the cone given by the conditions $n_{\varphi(r,a,0)}=0,\ 0< n_{\varphi(r,a,1)}<\ldots < n_{\varphi(r,a,r-1)}$.

\section{Main theorem}\label{section:main theorem}

We have defined a family of pairwise non-isomorphic representations of the McKay quiver on the Danilov
resolution $Y,$ which are $\theta$-stable with respect to stability parameters $\theta$ determined in Lemma~\ref{l:common stab}. The universal property of the moduli space $M_\theta(Q,R)$ will ensure that the Danilov resolution dominates one of its components.

\begin{definition}[Craw, Maclagan, Thomas]
Denote by $Y_\theta$ (for generic $\theta\in\Wt(Q)$) the unique irreducible component of the moduli
$M_\theta(Q,R),$ containing representations of the McKay quiver with all
arrows represented by a non-zero number.
Following~\cite[Theorem~4.3]{CrawMaclaganI}, we call $Y_\theta$ the
coherent component of $M_\theta(Q,R).$
\end{definition}

Note that representations of the McKay quiver with all arrows
represented by a non-zero number are $\theta$-stable under any
stability condition $\theta\in\Wt(Q).$ The coherent component is
reduced, irreducible, not-necessarily-normal toric variety of
dimension $3,$ projective over $X=\CC^3 /G$
(see~\cite[Theorem~4.3]{CrawMaclaganI}). Denote by $\pi_\theta$ the
corresponding projective, birational morphism:
\[\pi_\theta:Y_\theta\too X.\]  Denote by $\pi$ the natural toric morphism given by a sequence of toric
weighted blowups \[\pi:Y\too X,\] where $Y$ denotes the Danilov resolution.

Assume that the generic stability parameter $\theta$ is chosen such that all representations of the McKay
quiver in the family $\FFF$ on $Y$ are $\theta$-stable. Since $\theta$ is generic, by King~\cite[Proposition~5.3]{King}, there exists a universal family
of McKay quiver representations over $M_\theta(Q,R)$ and there exists a unique morphisms $\rho$
\[\rho:Y\too  M_\theta(Q,R).\]

Since the family $\FFF$ is defined by Logvinenko's \textit{reductor condition}, hence by \cite[Theorem~4.1, Definition~1.4]{Logvinenko:RIMS}
we see that that $\pi_\theta\circ\rho=\pi.$
\begin{theorem}[Main Theorem]\label{t:main theorem}
For any coprime natural numbers $a,r$ and any  rational
numbers $n_0,\ldots,n_{r-1}$ such that
\[n_{\varphi(r,a,0)}<\ldots<n_{\varphi(r,a,r-1)},\]
where $\varphi(r,a,j)=\modulo{\xi(r,a,j)-a}{r}$ and $\xi(r,a,\cdot)$ is an inverse of the permutation $\tau(r,a,\cdot)$ (see Definition~\ref{d:tau}), the Danilov resolution of the singularity of type $\frac{1}{r}(1,a,r-a)$ is normalization of the irreducible
component of the fine moduli space of $\theta$-stable
representations of the McKay quiver containing representations
corresponding to free orbits, for generic $\theta\in\Wt(Q),$ given by the condition
\[\theta(i)=n_i-n_{i+(r-a)}.\]
\end{theorem}

\begin{proof}
By Lemma~\ref{l:common stab}, every representation of the McKay quiver in the family $\FFF(r,a)$
 on the Danilov resolution,
constructed in Section~\ref{section:construction of family}, is $\ths$  for $\theta$ taken as above.
Therefore, there exists a unique morphisms $\rho$
\[\rho:Y\too  M_\theta(Q,R),\]
 and the following diagram commutes:
\[ \xymatrix{
  Y \ar[rr]^{\rho} \ar[dr]_{\pi}
                &  &    Y_\theta \ar[dl]^{\pi_\theta}    \\
                & X                 }\]

\noindent Morphism $\rho$ is proper since morphisms $\pi$ and
$\pi_\theta$ are proper
(see~\cite[Corollary~II.4.8.(e)]{Hartshorne:geometry}, $\pi_\theta$
is projective, hence separated). By
\cite[Exercise~II.4.4]{Hartshorne:geometry} the image of $\rho$ in
$Y_\theta$ is closed and is of dimension $3$ (see
Corollary~\ref{corollary:all reps non-iso}).

By the work of Craw, Maclagan, Thomas~\cite{CrawMaclaganI} the coherent
component $Y_\theta$ is a not-necessarily-normal
toric variety of dimension $3,$  hence $\rho$ is surjective onto
$Y_\theta.$ We are done if $\rho$ is injective, that is if two
representations of the McKay quiver corresponding to two distinct closed
points on the Danilov resolution are non-isomorphic. This is the
content of Corollary~\ref{corollary:all reps non-iso}.
\end{proof}

\begin{remark}
We note that the cone of stability conditions in the statement of Theorem~\ref{t:main theorem}
is top dimensional in the space $\Wt(Q).$
\end{remark}

\begin{definition}
A chamber of stability conditions is a connected component of the set of generic stability conditions (cf.~\cite{DolgachevHu},\cite{Thaddeus}).
\end{definition}

\begin{theorem}\label{th:second}
The closure of the cone defined by the condition of Theorem~\ref{t:main theorem} is a union of closures of chambers of stability conditions for the
action of the group $\PGL(Q,\CC)$ on the space $\Rep(Q,R).$
\end{theorem}

\begin{proof}
By proof of Lemma~\ref{l:common stab}, if some of the inequalities in the condition for $n_i$ is
not strict, then some representation $\FFF_j$ is strictly $\theta$-semistable. Conversely, if all inequalities are
strict, then all representations $\FFF_j,$ for $j=0,\ldots,r-1$ are $\theta$-stable.
\end{proof}

We note that there may be non-generic points on the other components of $\Rep(Q,R)$ which may define walls subdividing
the simplicial cone of stability conditions defined in Theorem~\ref{t:main theorem}.

\begin{example}
The permutation $\tau(5,2)$ is a cycle of length $5,$ namely
\[\tau(5,2)=(0,4,1,3,2).\] The sequence \[n_0<n_2<n_1<n_4<n_3\] implies that
the union of closures of chambers is given by the conditions $\theta_0+\ldots+\theta_4=0$ and
\[0\le\theta_2\le\theta_1+\theta_2+\theta_4\le\theta_2+\theta_4\le-\theta_0.\]

\end{example}

Using computer algebra packages, the author have checked that for small values of $a$ and $r$ and stability parameters as in the Main Theorem the moduli space of representation of the McKay quiver is normal. This suggests that the following holds.
\begin{conjecture}
The coherent component is normal in this case, that is, the Danilov resolution
is isomorphic to the coherent component $Y_\theta$ for any $\theta$ from Theorem~\ref{t:main theorem}.
\end{conjecture}

\end{document}